\documentclass[a4paper,12pt]{article}
\usepackage{amsmath}
\usepackage{amsthm,amssymb,amsfonts}
\usepackage{array}
\usepackage{color}
\usepackage{verbatim} 
\usepackage{cancel}

\newtheorem{theorem}{Theorem}[section]

\newtheorem{lemma}[theorem]{Lemma}
\newtheorem{proposition}[theorem]{Proposition}
\newtheorem{corollary}[theorem]{Corollary}
\newtheorem{remark}[theorem]{Remark}

\newtheorem{example}[theorem]{Example}

\title{Spectrum of composition operators on ${\mathcal S}({\mathbb R})$ with polynomial symbols}
\author{Carmen Fern\'andez, Antonio Galbis, Enrique Jord\'a\footnote{The present research was partially supported by the projects  MTM2016-76647-P, ACOMP/2015/186  (Spain). The third  author was partially supported by GVA, Project AICO/2016/054.}}

\begin{document}

\maketitle

\begin{abstract}
We study the spectrum of operators in the Schwartz space of rapidly decreasing functions which associate each function with its composition with a polynomial. In the case where this operator is mean ergodic we prove that its spectrum reduces to $\{0\},$ while the spectrum of any non mean ergodic composition operator with a polynomial always contains the closed unit disc except perhaps the origen. We obtain a complete description of the spectrum of the composition operator with a quadratic polynomial or a cubic polynomial with positive leading coefficient. 
\end{abstract}

\section{Introduction}

Composition operators on Fr\'echet spaces of smooth functions on the reals have attracted the attention of several authors recently (\cite{fernandez_galbis_jorda, galbis_jorda, kw,adam1,adam2, adam2.5}) but to our knowledge very little is known about the spectra of composition operators in this setting. See for instance \cite{bonet_domanski} or \cite{albanese_bonet_werner}, where the spectrum of composition operators and other classical operators on spaces of real smooth functions is investigated. This contrasts with the large number of existing articles studying spectral properties of composition operators in Banach spaces of analytic functions on the unit disc.
\par\medskip
We study the spectrum of composition operators defined in the Schwartz space $\mathcal{S}({\mathbb R})$ of smooth rapidly decreasing functions, $C_\varphi:\mathcal{S}({\mathbb R})\to \mathcal{S}({\mathbb R}), f\mapsto f\circ \varphi,$ in the case that $\varphi$ is a non constant polynomial. A smooth function $\varphi:{\mathbb R}\to {\mathbb R}$ is said to be a symbol for $\mathcal{S}({\mathbb R})$ if $C_\varphi$ maps $\mathcal{S}({\mathbb R})$ continuously into itself. The symbols for $\mathcal{S}({\mathbb R})$ were completely characterized in \cite[Theorem 2.3]{galbis_jorda}. It follows from that characterization that any non constant polynomial $\varphi$ is a symbol for $\mathcal{S}({\mathbb R}).$ The results of the present paper complement our study in \cite{fernandez_galbis_jorda} where we investigate dynamics and the spectrum of some particular composition operators. Concrete examples were given where the spectrum coincides with the open unit disc, the unit circle or ${\mathbb C}\setminus \{0\}.$ In 
particular, the spectrum of translation and dilation operators was analyzed in \cite[Examples 5-6]{fernandez_galbis_jorda}:
\begin{example}{\rm \begin{itemize}
                     \item[(a)] Let $\varphi(x) = x+1.$  Then $\sigma\left(C_\varphi\right) = \{\lambda\in {\mathbb C}:\ |\lambda| = 1\}.$
\item[(b)] Let $\varphi(x) = ax$ where $a\neq 0$ and $|a|\neq 1.$ Then $\sigma\left(C_\varphi\right) = {\mathbb C}\setminus\{0\}.$
                    \end{itemize}

} 
\end{example}
However, in \cite{fernandez_galbis_jorda} we did not obtain any result concerning the spectrum of the composition operator with a polynomial of degree greater than one. This is precisely the objective of this work. It turns out that some  dynamical properties of the composition operator are characterized by the spectrum of the operator. For example, the spectrum of a mean ergodic composition operator is always contained in the closed unit disk (\cite[Corollary 4.5]{fernandez_galbis_jorda}). This result can be improved when the symbol is a polynomial. As we prove in Theorem \ref{th:withoutfixedpoints}, the spectrum of a composition operator which is mean ergodic and whose symbol is a polynomial with degree greater than one coincides with $\{0\}$ while the behavior of non mean ergodic composition operators with polynomial symbols is different (Theorem \ref{th:polynomial_fixed_point}). For strictly decreasing symbols, not necessarily polynomials, the containment of $\partial {\mathbb D}$ in the spectrum of the 
operator is equivalent to its mean ergodicity. 
\par\medskip
In \cite{bonet_domanski} the spectrum of composition operators on ${\mathcal A}({\mathbb R})$, the space of all real analytic functions, is investigated for the case that the symbol is a quadratic polynomial. For quadratic polynomials, we have a complete characterization of the spectrum of the corresponding composition operator depending on the number of fixed points of the polynomial. As expected, the spectrum of the operator depends on the space where it is considered. To give an example, when $\varphi$ is a quadratic polynomial without fixed points then $\sigma_{{\mathcal A}({\mathbb R})}(C_\varphi) = {\mathbb C},$ whereas $\sigma_{{\mathcal S}({\mathbb R})}(C_\varphi) = \{0\}.$
\par\medskip
Theorem \ref{th:cubic-positive} contains a complete description of the spectrum of a composition operator with a polynomial of degree three whose leading coefficient is positive. For polynomials with negative leading coefficient some partial results are available but we lack a complete characterization.
\par\medskip
The final section contains some results concerning the spectra of composition operators with (non polynomial) monotone symbols.
\par\medskip
We recall that given an operator $T:X\to X$ on a Fr\'echet space $X,$ $\sigma(T),$ the spectrum of $T,$ is the set of all $\lambda\in {\mathbb C}$ such that $T - \lambda I:X\to X$ does not admit a continuous linear inverse. $T$ is said to be power bounded if $\{T^n(x):\ n\in {\mathbb N}\}$ is bounded for each $x\in X.$ A closely related concept to power boundedness is that of {\em mean ergodicity}. Given $T\in L(X)$, the Ces\`aro means of $T$ are defined as $T_{[n]}=\sum_{k=1}^{n}T^k/n$. $T$ is said to be mean ergodic when $T_{[n]}$ converges to an operator $P$, which is always a projection, in the strong operator topology, i.e. if $(T_{[n]}(x))$ is convergent to $P(x)$ for each $x\in X$. 
\par\medskip
From now on $\varphi_n = \varphi\circ\ldots\circ\varphi$ denotes the $n$-th iteration of $\varphi.$ 
\par\medskip 
The following results will be used in what follows.

\begin{lemma}{\rm \cite[3.10]{fernandez_galbis_jorda}\label{lem:cota-iteradas} Let $\varphi$ be a polynomial of even degree without fixed points. Then there is $N\in {\mathbb N}$ such that
$\psi = \varphi_{N}$ has neither zeros nor fixed points. Moreover, for every $K > 0$ there is $m_0\in {\mathbb N}$ such that
$$
\left|\psi_{m+1}(t)\right| \geq K \left(\psi_{m}(t)\right)^2\ \ \forall m\geq m_0,\ \ \ \forall t\in {\mathbb R}.
$$
}
\end{lemma}

\begin{theorem}{\rm \cite[3.11]{fernandez_galbis_jorda} Let $\varphi$ be a polynomial with degree greater than or equal to two. Then, the following are equivalent:
\begin{itemize}
\item[(1)] $C_\varphi$ is power bounded.
\item[(2)] $C_\varphi$ is mean ergodic.
\item[(3)] The degree of $\varphi$ is even and it has no fixed points.
\end{itemize}
} 
\end{theorem}

\section{Polynomial symbols}

Two polynomials $\varphi, \psi\in {\mathbb R}[x]$ are {\it linearly equivalent} if there exists $\ell(x) = ax + b$ with $a,b\in {\mathbb R}$ and $a\neq 0$ such that $\psi = \ell^{-1}\circ \varphi\circ \ell.$ Then, for every $\lambda\in {\mathbb C},$
$$
C_\varphi - \lambda I = C_\ell^{-1}\circ\left(C_\psi - \lambda I\right)\circ C_\ell,
$$ from where it follows that $\sigma\left(C_\psi\right) = \sigma\left(C_\varphi\right).$

\par\medskip
The first result follows immediately from this observation and \cite[Examples 5-6]{fernandez_galbis_jorda} as each polynomial of degree one other than the identity is linearly equivalent to a translation or to a dilation.

\begin{proposition}{\rm Let $\varphi(x) = ax + b$ a polynomial with $a,b\in {\mathbb R}$ and $a\neq 0.$ Then
\begin{itemize}
 \item[(a)] For $a\neq 1,$ $\varphi$ is linearly equivalent to $\psi(x) = ax.$  Hence $\sigma\left(C_\varphi\right) = {\mathbb C}\setminus\left\{0\right\}$ for $a\neq \pm 1$ while $\sigma\left(C_\varphi\right) = \left\{-1,1\right\}$ for $a = -1.$
\item[(b)] For $a = 1$ and $b\neq 0,$ $\varphi$ is linearly equivalent to $\psi(x) = x + 1.$ Hence $\sigma\left(C_\varphi\right) = \left\{\lambda\in {\mathbb C}: |\lambda| = 1\right\}.$
\end{itemize}
}
\end{proposition}

From now on we will consider only polynomials of degree greater than one.
\par\medskip
We observe that the following version of the spectral theorem holds in our setting. 

\begin{proposition}{\rm For every symbol $\varphi$ and $N\in {\mathbb N},$
$$
\sigma(C_{\varphi_N}) = \left\{\lambda^N\in {\mathbb C}:\ \lambda\in \sigma(C_{\varphi})\right\}.
$$
} 
\end{proposition}
\begin{proof}
 For every $\mu\in {\mathbb C}\setminus\{0\}$ let $\lambda_1,\ldots, \lambda_N$ denote its $N$-roots. Then 
$$
C_{\varphi_N} - \mu I = C_\varphi^N - \mu I = \left(C_\varphi - \lambda_1 I\right)\cdot \cdot \cdot \left(C_\varphi - \lambda_n I\right),
$$ from where the conclusion follows.
\end{proof}

We also recall the following elementary properties, which will be used in what follows. 

\begin{proposition}\label{prop:elementary}{\rm Let $\varphi$ be a symbol for ${\mathcal S}({\mathbb R}).$ 
\begin{itemize}
\item[(a)] If $\varphi$ admits fixed points then $1\in \sigma(C_\varphi).$
\item[(b)] If $a$ is a fixed point of $\varphi$ then $\varphi'(a)\in \sigma(C_\varphi).$ 
\item[(c)] If $\varphi$ is a polynomial then $0\in \sigma(C_\varphi)$ if and only if $\varphi'$ vanishes at some point.
 \end{itemize}
}
\end{proposition}
\begin{proof}
(a) All the functions in the range of $C_\varphi - I$ vanish at fixed points of $\varphi,$ hence the conclusion.
\par\medskip
(b) In fact, the derivative of any function in the range of $C_\varphi - \varphi'(a)I$ vanishes at point $a,$ hence $C_\varphi - \varphi'(a)I$ is not surjective.
\par\medskip
(c) If $\varphi'$ does not vanish then $\inf_{x\in {\mathbb R}}|\varphi'(x)| > 0$ and $C_\varphi$ is surjective by \cite[4.2]{galbis_jorda}. Moreover $C_\varphi$ is injective as $\varphi({\mathbb R}) = {\mathbb R}.$ Conversely, if $\varphi'(x_0) = 0$ then the derivative of any function in the range of $C_\varphi$ vanishes at point $x_0.$ Hence $0\in \sigma(C_\varphi).$
\end{proof}

\par\medskip
We observe that an arbitrary symbol $\varphi$ for ${\mathcal S}({\mathbb R})$ satisfies conditions (i) and (ii) in the next lemma with $r = 1$ if, and only if, $C_\varphi$ is power bounded. This is the content of \cite[Proposition 3.9]{fernandez_galbis_jorda}.

\begin{lemma}\label{lem:convergencia_serie}{\rm Let $\varphi$ be a polynomial of degree $\geq 2$. Assume that for all $r>1$, $n \in{\mathbb N}$ there exist $C>0$, $q\in{\mathbb N}$ such that the following conditions hold for each $x\in{\mathbb R}$ and $m\in {\mathbb N}$:
	
	\begin{itemize}
	\item[(i)]  $|\varphi_{m}^{(n)}(x)| \leq Cr^m(1+|\varphi_m(x)|)^q$
	\item[(ii)] $|x|\leq (1+|\varphi_m(x)|)^q$.
	\end{itemize}
	
	 Then the series 
	
	$$\sum_m \mu^m f\circ \varphi_m$$
	
	\noindent is convergent in $S({\mathbb R})$ for each $|\mu|<1$. If in addition we assume that (i) happens with $r=1$ then the series is convergent for each $\mu\in{\mathbb C}$.
}	
\end{lemma}	

\begin{proof}
For each $n,m\in {\mathbb N}_0,$  by Fa\'a de Bruno formula
$$
(f\circ \varphi_m)^{(n)}(x) = \sum_{j=1}^{n} f^{(j)}(\varphi_m(x))B_{n,j}\left(\varphi_m'(x)\ldots \varphi_{m}^{(n-j+1)}(x)\right),
$$ where $B_{n,j}$ are the Bell polynomials. Thus, from (i) and (ii), given $f\in {\mathcal S}({\mathbb R}),$ $n\in {\mathbb N}_0, r > 1$ and a polynomial $P$ we find another polynomial $\tilde{P}$ such that 
\begin{equation}\label{eq:Ptilde}
 \left|P(x) (f\circ \varphi_m)^{(n)}(x)\right|\leq r^m\left|\tilde{P}(\varphi_m(x))\sum_{j=1}^{n} f^{(j)}(\varphi_m(x)) \right|.
\end{equation}
Since $f\in {\mathcal S}({\mathbb R})$ there is $M > 0$ such that (\ref{eq:Ptilde}) can be estimated by $M r^m.$

Therefore, given $\mu\in {\mathbb D}$ we choose $r > 1$ such that $|r\mu| < 1$ and the convergence in ${\mathcal S}({\mathbb R})$ of $\sum \mu^m f\circ \varphi_m$ follows.
\par\medskip
If in addition (i) is satisfied with $r = 1$ then $C_{\varphi}$ is power bounded \cite[Proposition 3.9]{fernandez_galbis_jorda} and we have the estimate (\ref{eq:Ptilde}) with $r = 1.$ By \cite[Theorem 3.1]{fernandez_galbis_jorda} $\varphi$ has even degree and lacks fixed points. 
\par\medskip
First we assume that 
 \begin{itemize}
  \item[(a)] $|\varphi(x)| > x^2$ for every $x\in {\mathbb R}$ \par and
\item[(b)] $\inf\left\{|\varphi(x)|:\ x\in {\mathbb R}\right\} = a > 1.$
 \end{itemize}
This implies that $\left|\varphi_m(x)\right| > a^{2^{m-1}}$ and the series 
$$
\sum_m \frac{\mu^m}{\left|\varphi_m(x)\right|}
$$ converges absolutely and uniformly in ${\mathbb R}$ for every $\mu \in {\mathbb C}.$ Using (\ref{eq:Ptilde}) with $r = 1$ we immediately have, for every polynomial $P$ and $n\in {\mathbb N},$ that 
$$
\begin{array}{*2{>{\displaystyle}l}}
\left|\mu^m P(x) (f\circ \varphi_m)^{(n)}(x)\right|& \leq\left|\frac{\mu^m}{1+|\varphi_m(x)|}(1+|\varphi_m(x)|)\tilde{P}(|\varphi_m(x)|)\sum_{j=1}^{n} f^{(j)}(\varphi_m(x))\right| \\ & \\ & 
\leq M\frac{|\mu|^m}{1+|\varphi_m(x)|}
\end{array}
$$ for some $M > 0.$ Consequently 
$$
\sum_m \mu^m f\circ \varphi_m
$$ converges in ${\mathcal S}({\mathbb R})$ for each $\mu \in {\mathbb C}$ and $f\in{\mathcal S}({\mathbb R}).$
\par\medskip
In the general case, we may find $N\in {\mathbb N}$ such that $\varphi_{N}$ satisfies conditions (a) and (b) (\cite[Lemma 3.10]{fernandez_galbis_jorda}). Hence
$$
\sum_m \mu^m f\circ \varphi_m = \sum_{j=0}^{N-1}\mu^j \left(\sum_m \left(\mu^{N}\right)^m \left(f \circ \varphi_j\right)\circ \varphi_{Nm}\right)
$$ converges in ${\mathcal S}({\mathbb R}).$
\end{proof}

\begin{theorem}\label{th:withoutfixedpoints}{\rm Let $\varphi$ be a polynomial with even degree and without fixed points. Then $\sigma(C_\varphi)=\{0\}.$
}
\end{theorem}
\begin{proof}
From Lemma \ref{lem:cota-iteradas} we find $N\in {\mathbb N}$ such that if $\psi=\varphi_N,$ $$\min\{|\psi(x)|:x \in {\mathbb R}\}=a>1$$ and $$
\left|\psi_{m+1}(x)\right| \geq  \left(\psi_{m}(x)\right)^2\ \ \forall m,\ \ \ \forall x\in {\mathbb R}.
$$
In particular, this gives \begin{equation}\label{eq:lowerbound}\left|\psi_{m}(x)\right|>a^{2^{m-1}},\end{equation} for all $x$ and every $m.$
\par\medskip
Since the range of $\varphi$ is a proper (unbounded) interval then $C_\varphi$ is not injective and $0\in \sigma(C_\varphi).$ To finish the proof it suffices to show that $\sigma(C_\psi)\subset \{0\}.$

To this end, we fix $\lambda\in {\mathbb C}, \lambda\neq 0,$ and check that $C_\psi - \lambda I$ is a bijection, hence a topological isomorphism by the open mapping theorem.
\par
(i) Injectivity. Let us assume $C_\psi f = \lambda f$ for some $f\in {\mathcal S}({\mathbb R}).$ Then, for every $m\in {\mathbb N},$
$$
f(x) = \lambda^{-m}f\left(\psi_m(x)\right) = \frac{\psi_m(x)f\left(\psi_m(x)\right)}{\lambda^m \psi_m(x).}
$$ Since
$$
\left|\lambda^m \psi_m(x)\right| \geq |\lambda|^m\cdot a^{2^{m-1}}\to \infty
$$ then $f(x) = 0$ for all $x\in {\mathbb R}.$
\par
(ii) Surjectivity. From \cite[Theorem 3.11 and Proposition 3.9]{fernandez_galbis_jorda} $C_\varphi$ is power bounded and the hypothesis in Lemma \ref{lem:convergencia_serie} are satisfied with $r = 1.$ Then, for every $g\in {\mathcal S}({\mathbb R})$ and $\lambda\neq 0$ the series 
\begin{equation}\label{eq:f}
f = -\sum_{k=0}^\infty \frac{1}{\lambda^{k+1}}g\circ\psi_k
\end{equation} converges in ${\mathcal S}({\mathbb R})$ and clearly $C_\psi f - \lambda f = g.$
\end{proof}

According to Proposition \ref{prop:elementary} the behavior of composition operators with polynomials having fixed points is different. In order to obtain more information we first we need an auxiliary result.
 
\begin{lemma}\label{lem:fast}{\rm Let $\varphi$ be a polynomial with degree greater than one. Then, there is $M > 0$ such that for each $|x|>M,$  $$\lim_n \frac{1}{\lambda^n}f(\varphi_n(x))=0,$$
 for $0<|\lambda|\leq 1$ and every $f\in {\mathcal S}({\mathbb R}).$
 }
 \end{lemma}
 \begin{proof} 
We fix $2 > p > 1$ and take $M > 1$ such that $\left|\varphi(x)\right| > |x|^p$ whenever $|x| > M.$ Then $|x| > M$ implies 
$$
\left|\varphi_n(x)\right| > M^{p^n}\ \ \forall n\in {\mathbb N}.
$$ Finally
$$
\lim_n \left|\frac{1}{\lambda^n}f(\varphi_n(x))\right| \leq \lim_n \left|\frac{\varphi_n(x) f(\varphi_n(x))}{\lambda^n M^{p^n}}\right| = 0.
$$
\end{proof}

\begin{lemma}\label{lem:negativederivative}{\rm Let $\varphi$ be a polynomial with odd degree greater than one such that $\displaystyle\lim_{x\to +\infty}\varphi(x)= -\infty.$ Let $a$ be a fixed point of $\varphi.$ If $a$ is the largest fixed point of $\varphi_2$ then $\varphi'(a)\leq -1.$
}
\end{lemma}
\begin{proof}
 We first observe that $\displaystyle\lim_{x\to +\infty}\varphi_2(x)= +\infty.$ From $\varphi_2(x) > x$ for every $x > a$ we get
$$
\left(\varphi'(a)\right)^2 = \varphi_2'(a) = \lim_{x\to a}\frac{\varphi_2(x)-a}{x-a} \geq 1.
$$ Since $a$ is also the largest fixed point of $\varphi$ then $\varphi(x) < x$ for all $x > a,$ from where it follows
$$
\varphi'(a) = \lim_{x\to a}\frac{\varphi(x)-a}{x-a}\leq 1.
$$ Consequently $\varphi'(a)\leq -1$ or $\varphi'(a) = 1.$ Finally we check that $\varphi'(a)\leq 0.$ Otherwise there is $\delta > 0$ such that $\varphi$ is strictly increasing on $[a, a+\delta].$ Since $a < \varphi(a+\delta) \leq a+\delta$ then $\varphi\left([a,a+\delta]\right) \subset [a,a+\delta].$ This is a contradiction. In fact, for every $x > a$ the sequence $\left(\varphi_{2n}(x)\right)_n$ is increasing and unbounded.
\end{proof}

\begin{theorem}\label{th:polynomial_fixed_point} {\rm Let $\varphi$ be a polynomial with degree greater than one and having fixed points. Then,
$$
\overline{{\mathbb D}}\setminus \{0\}\subset \sigma(C_\varphi).$$
}
\end{theorem}
\begin{proof} (a) First we consider the case that $\displaystyle \lim_{x\to +\infty}\varphi(x) = +\infty.$ We fix $\lambda \in \overline{{\mathbb D}}\setminus \{0\}$ and assume that $\lambda \notin \sigma(C_\varphi),$ hence $\lambda \neq 1.$ Let $a\in {\mathbb R}$ be given with the property that $\varphi(a)=a$ and $\varphi(x)>x$ for $x>a.$ Since $\varphi'(a)\geq 1$ then $\varphi$ is strictly increasing in some interval $[a, a+\delta].$ Let $\psi:[a, \varphi(a+\delta)]\to [a, a+\delta]$ be the inverse of $\varphi:[a, a+\delta] \to [a, \varphi(a+\delta)].$ Since $\varphi(a+\delta) > a+\delta$ then $\psi_k,$ the $k$-th iterate of $\psi,$ is well defined for every $k\in {\mathbb N}.$
\par\medskip
We fix $x_0\in (a, a+\delta)$ and define $x_{k} = \psi_k(x_0).$ Then $\left(x_k\right)_k$ is a decreasing sequence converging to $a.$ Let $J_0$ be a closed interval contained in $(x_1, x_0)$ and take a smooth function $g$ whose support is contained in $(x_1, x_0)$ and satisfying $g(x) = 1$ for all $x\in J_0.$ Then, there is a unique $f\in {\mathcal S}({\mathbb R})$ such that
\begin{equation}\label{eq:pr1}
f\left(\varphi(x)\right) - \lambda f(x) = g(x),\ \ x\in {\mathbb R}.\end{equation} After iterating this identity we obtain
\begin{equation}\label{eq:pr2}
 f\left(\varphi_n(x)\right) = \lambda^{n} f(x) + \sum_{k=0}^{n-1}\lambda^{n-1-k}g\left(\varphi_k(x)\right).
\end{equation} Let $M > 0$ be as in Lemma \ref{lem:fast}. For each $x > a$ the sequence $\left(\varphi_n(x)\right)_n$ diverges to infinity, so there is $m\in {\mathbb N}$ with $\varphi_n(x) > M$ and it easily follows that $$\lim_n \frac{1}{\lambda^n}f(\varphi_n(x))=0.$$ We conclude
\begin{equation}\label{eq:pr3}
f(x)=-\sum_{k=0}^\infty \frac{1}{\lambda^{k+1}}g(\varphi_k(x))\end{equation} for all $x>a.$ Finally, we fix $y_0\in J_0$ and  define $y_k = \psi_k(y_0)\in (x_{k+1}, x_k).$ We have $\varphi_m(y_m) = y_0\in J_0,$ while $\varphi_k(y_m) = \varphi_{k-m}(y_0) > x_0$ for $k > m$ and $\varphi_k(y_m) = \psi_{m-k}(y_0)<x_1$ for $k < m.$ Consequently
$$
f(y_m) = -\lambda^{-m-1}\ \ \mbox{while}\ \ f(a) = 0.
$$ The last identity follows from (\ref{eq:pr1}) using $\lambda \neq 1.$ Since
$$
\lim_m\left|f(y_m)\right| \neq \left|f(a)\right|
$$ we get a contradiction.
\par\medskip
(b) To deal with the case that $\displaystyle \lim_{x\to +\infty}\varphi(x) = -\infty$ we have to consider two possibilities, depending on whether the degree of the polynomial is even or odd.
\par
(i) First case: the degree of $\varphi$ is even. Since $\varphi$ is linearly equivalent to $\psi(x) = -\varphi(-x)$ and $\displaystyle \lim_{x\to +\infty}\psi(x) = +\infty$ then
$$
\overline{{\mathbb D}}\setminus \{0\}\subset \sigma(C_\psi) = \sigma(C_\varphi)
$$ and we are done.
\par
(ii) Second case: the degree of $\varphi$ is odd. Then $\displaystyle\lim_{x\to +\infty}\varphi_2(x) = +\infty.$ As above, $\varphi_2$ is strictly increasing in some interval $[a, a+\delta],$ where $a$ is the greatest fixed point of $\varphi_2.$ Moreover, we can take $\delta$ small enough so that $\varphi(x) < a$ for every $x\in (a, a+\delta].$ This is obvious in the case that $\varphi(a) < a.$ Otherwise, $\varphi(a) = a$ and $\varphi'(a)<0$ (Lemma \ref{lem:negativederivative}) and we can take $\delta$ so that $\varphi$ is decreasing on $[a, a+\delta],$ hence $\varphi(x) < \varphi(a) = a$ for every $x\in (a, a+\delta].$ Now, we denote by $\psi$ the inverse of $\varphi_2:[a, a+\delta]\to [a, \varphi_2(a+\delta)].$

Proceeding as in (a), we fix $x_0\in (a, a+\delta)$ and define $x_{k} = \psi_k(x_0).$ Let $J_0$ be a closed interval contained in $(x_1, x_0)$ and take a compactly supported smooth function $g$ whose support is contained in $(x_1, x_0)$ and satisfying $g(x) = 1$ for all $x\in J_0.$ As in (a), there is $f\in {\mathcal S}({\mathbb R})$ such that equation (\ref{eq:pr3}) holds for $x > a.$
 Finally, we fix $y_0\in J_0$ and  define $y_k = \psi_k(y_0)\in (x_{k+1}, x_k).$ We have $\varphi_{2m}(y_m) = y_0\in J_0,$ $\varphi_{2k}(y_m) > x_0$ for $k > m,$ $\varphi_{2k}(y_m) < x_1$ for $k < m.$ Moreover $\varphi_{2k+1}(y_m)\notin (a, a+\delta]$ since otherwise $\varphi_{2k+2}(y_m) < a,$ which is a contradiction. Consequently
$$
f(y_m) = -\lambda^{-2m-1}\ \ \mbox{while}\ \ f(a) = 0.
$$ The same argument as in case (a) gives a contradiction.
\end{proof}

\begin{corollary}{\rm Let $\varphi$ be a polynomial of degree greater than one. Then $C_\varphi$ is mean ergodic if and only if $\sigma(C_\varphi) = \{0\}.$
 }
\end{corollary}
\begin{proof}
 Apply \cite[Theorem 3.11]{fernandez_galbis_jorda} and Theorems \ref{th:withoutfixedpoints} and \ref{th:polynomial_fixed_point}.
\end{proof}

\begin{theorem}\label{th:simplefixedpoint}{\rm Let $\varphi$ be a polynomial of degree greater than one and having a fixed point $a$ such that $\varphi'(a) > 1$ and $\varphi^{(n)}(a)\geq 0$ for all $n\geq 2.$ Then
$$
{\mathbb C}\setminus \{0\} \subset \sigma\left(C_\varphi\right).
$$
 }
\end{theorem}
\begin{proof} Since $$\varphi(x) = a + \sum_{n=1}^\infty \frac{\varphi^{(n)}(a)}{n!}(x-a)^n$$ then $\varphi$ and all its derivatives are increasing in $[a, +\infty).$ An inductive argument using Fa\`a di Bruno formula implies that also $\varphi_k^{(n)}$ is increasing in $[a, +\infty)$ for every $k, n\in {\mathbb N}_0.$ We observe that $\varphi(x) > a + \varphi'(a)(x-a) > x$ for any $x > a.$ Hence $a$ is the largest fixed point of $\varphi.$
\par\medskip
We already know that $\overline{{\mathbb D}}\setminus \{0\} \subset \sigma\left(C_\varphi\right).$ We now fix $|\lambda| > 1$ and assume that $\lambda\notin \sigma(C_\varphi).$ We fix $x_0 > a$ and define $x_{k+1} = \psi(x_k),$ where $\psi$ stands for the inverse of $\varphi:[a,+\infty)\to [a,+\infty).$ Then $\left(x_k\right)_k$ is a decreasing sequence converging to $a.$ We put $I_k = \left(x_{k+1}, x_k\right),$ so that $I_k = \psi_k(I_0),$ and let $J_0$ be a closed subinterval of $I_0$ and $J_k:=\psi_k(J_0).$ Finally, we consider a compactly supported smooth function $g$ whose support is contained in $I_0$ and such that $g(x) = x$ for every $x\in J_0.$ Then there is $f\in {\mathcal S}({\mathbb R})$ satisfying $C_\varphi f -\lambda f=g.$ Hence
$$
f\left(\varphi_n(x)\right) = \lambda^{n} f(x) + \sum_{k=0}^{n-1}\lambda^{n-1-k}g\left(\varphi_k(x)\right) \ \ \forall n\in {\mathbb N}, x\in {\mathbb R}.
$$ Since $|\lambda| > 1$ and $\left(f(\varphi_n(x))\right)_n$ is a bounded sequence then
$$
f(x) = -\frac{1}{\lambda}\sum_{j=0}^{\infty}\lambda^{-j}g\left(\varphi_j(x)\right)\ \ \forall x\in {\mathbb R}.
$$
For every $x\in J_k$ we have $\varphi_k(x)\in J_0$ while $\varphi_n(x)\notin I_0$ for every $n\neq k.$ Consequently
$$
f(x) = -\lambda^{-k-1}g\left(\varphi_k(x)\right) = -\frac{\varphi_k(x)}{\lambda^{k+1}}\ \ \forall x\in J_k.
$$ In order to obtain a contradiction we proceed as follows. Our hypothesis and Fa\`a di Bruno formula permit to conclude
\begin{equation}\label{eq:estimacioninferior}
\varphi_{m+1}^{(n)}(a) \geq \varphi_{m}^{(n)}(a)\cdot \varphi'(a)^n.
\end{equation} We select $n_0\in {\mathbb N}$ so that $\varphi'(a)^{n_0} > |\lambda|.$ We can find $n > n_0$ and $m\in {\mathbb N}$ such that $\varphi_m^{(n)}(a)\neq 0.$ Otherwise, every iterate $\varphi_m$ would have degree less than or equal $n_0,$ which is a contradiction. From (\ref{eq:estimacioninferior}) we get
$$
\varphi_{k+m}^{(n)}(a) \geq \varphi_{m}^{(n)}(a)\cdot \varphi'(a)^{kn}\ \ \forall k\in {\mathbb N}.
$$ Finally, for every $x\in J_{k+m}$ we obtain, with $C = \frac{\varphi_{m}^{(n)}(a)}{|\lambda|^{1+m}},$
$$
\begin{array}{*2{>{\displaystyle}l}}
 \left|f^{(n)}(x)\right| & = \frac{\varphi_{k+m}^{(n)}(x)}{|\lambda|^{k+m+1}} \geq \frac{\varphi_{k+m}^{(n)}(a)}{|\lambda|^{k+m+1}}\\ & \\ & \geq C \left(\frac{\varphi'(a)^{n}}{|\lambda|}\right)^k.
\end{array}
$$ We conclude that $f^{(n)}$ is not a bounded function, which is a contradiction.
\end{proof}

The following result will be useful in the proof of Proposition \ref{prop:negative-odd}.

\begin{proposition}\label{rem:composition}{\rm Let $\eta$ be a polynomial with odd degree and negative leading coefficient such that $\varphi = \eta\circ \eta$ satisfies the hypothesis in Theorem \ref{th:simplefixedpoint}. Then ${\mathbb C}\setminus\{0\}\subset \sigma(C_\eta).$
}
\end{proposition}
\begin{proof}
Since $\eta$ has odd degree then it has fixed points and we can apply Theorem \ref{th:polynomial_fixed_point} to get $\overline{{\mathbb D}}\setminus \{0\} \subset \sigma\left(C_\eta\right).$ We now fix $|\lambda| > 1$ and assume that $\lambda\notin \sigma(C_\eta).$ We observe that $a$ is the largest fixed point of $\varphi.$ From Lemma \ref{lem:negativederivative} we get $\delta > 0$ such that $\eta(x) < a$ for all $x\in (a, a+\delta).$ Now we fix $x_0\in (a, a+\delta)$ and define $g$ as in the proof of Theorem \ref{th:simplefixedpoint}. Then there is $f\in {\mathcal S}({\mathbb R})$ such that
$$
f(x) = -\frac{1}{\lambda}\sum_{n=0}^{\infty}\lambda^{-n}g\left(\eta_n(x)\right)\ \ \forall x\in {\mathbb R}.
$$ We observe that $\eta_{2j} = \varphi_j$ and $g\left(\eta_{2j+1}(x)\right) = 0$ for all $j\in {\mathbb N}_0$ and $x\geq a$ (otherwise $\eta_{2j+1}(x)\in (a, a+\delta)$ and $\eta_{2j+2}(x) < a,$ which is a contradiction). Then
$$
f(x) = -\frac{1}{\lambda}\sum_{j=0}^{\infty}\lambda^{-2j}g\left(\varphi_j(x)\right)\ \ \forall x\geq a.
$$ Now we proceed as in Theorem \ref{th:simplefixedpoint} to get a contradiction.
\end{proof}

As an application of Theorem \ref{th:simplefixedpoint} and Proposition \ref{prop:elementary} we have the following.

\begin{example}{\rm Let $\varphi(x) = x^p,$ $\, p\geq 2.$ Then $\sigma(C_\varphi) = {\mathbb C}.$
}
\end{example}

\begin{example}\label{ex:simplefixedpoint}{\rm Let $\varphi$ be a polynomial of degree $N>1$ with positive leading coefficient and complex fixed points $z_1,\dots, z_N$ such that 
\begin{enumerate}
 \item $z_N \in {\mathbb R}$ and the multiplicity of $z_N$ as a fixed point is 1.
\item $Re(z_k)\leq z_N$ for $k<N.$
\end{enumerate}
Then ${\mathbb C}\setminus\{0\}\subset \sigma(C_\varphi).$
}
\end{example}

In fact, we can apply Theorem \ref{th:simplefixedpoint} taking $a = z_N.$

\section{Quadratic polynomials}
\par\medskip
Next we apply the previous results to discuss the spectrum of $C_\varphi$ in the case that $\varphi$ is a quadratic polynomial. Such a polynomial $\varphi(x) = a_0 + a_1x + a_2 x^2$ ($a_2 \neq 0$) is linearly equivalent to $\psi(x) = x^2 + c$ where $c = a_0a_2 + \frac{a_1}{2} - \frac{a_1^2}{4}.$ In fact, take $\ell(x) = ax + b$ where $a = a_2, b = \frac{a_1}{2}.$ It is routine to check that $\varphi = \ell^{-1}\circ \psi\circ \ell.$ We observe that $0\in \sigma(C_\psi) = \sigma(C_\varphi)$ since the range of $C_\psi$ consists of even functions.
\par\medskip
$c > \frac{1}{4}$ implies that $\varphi$ and $\psi$ lack fixed points, hence $\sigma(C_\varphi) = \{0\}$ (Theorem \ref{th:withoutfixedpoints}). In the case $c < \frac{1}{4}$ we have that $\varphi$ (and also $\psi$) has two different fixed points and we can apply Theorem \ref{th:simplefixedpoint} (see also Example \ref{ex:simplefixedpoint}) to conclude that $\sigma(C_\varphi) = {\mathbb C}.$
\par\medskip
Our next aim is to discuss the case $c = \frac{1}{4}.$

\begin{lemma}\label{lem:quadratic1}{\rm Let $\varphi(x) = x^2 + \frac{1}{4}$ be given. Then, for every $r > 1$ there exist $C > 0$ and $p\in {\mathbb N}$ such that
$$
 \left|\varphi_m'(x)\right| \leq C r^m\left(1+\varphi_m(x)\right)^p.
 $$
 }
\end{lemma}
\begin{proof}
 We first observe that $|\varphi'(x)| \leq 2\varphi(x)$ with equality for $x = \pm\frac{1}{2}$ and also
 \begin{equation}\label{eq:recurrence-derivative}
  \varphi_{m+1}'(x) = 2\varphi_m(x)\varphi_m'(x).
 \end{equation} Proceeding by recurrence we conclude
 \begin{equation}\label{eq:derivative-as-a-product}
  \varphi_m'(x) = 2^m\prod_{j=0}^{m-1}\varphi_j(x),
 \end{equation} where $\varphi_0(x) = x.$
\par\medskip
 Since every $\varphi_m$ ($m\geq 1$) is even and $\varphi_m'$ is odd we only need to consider the case $x\geq 0.$ Now we proceed in several steps.
 \par\medskip
(i) For $0\leq x\leq \frac{1}{2}$ we have $\varphi_m(x)\leq \varphi_m(\frac{1}{2}) = \frac{1}{2}.$ An induction argument gives
$$
\varphi_m'(x)\leq \varphi_m'(\frac{1}{2}) = 1.
$$
\par\medskip
(ii) For $x\geq x_0:= 1 + \frac{\sqrt{3}}{2}$ we have $\varphi'(x)\leq \varphi(x).$ Since $\varphi(x) > x$ then also $2\varphi_m(x) \leq \varphi_{m+1}(x)$ for every $m\in {\mathbb N}.$ We check that
$$
\varphi_m'(x) \leq \varphi_m^2(x)\ \ \forall m\in {\mathbb N},\ x\geq x_0.
$$ In fact, this inequality is obvious for $m = 1$ and assuming that it is true for $m$ we obtain
$$
\varphi_{m+1}'(x) = 2\varphi_m(x)\varphi_m'(x)\leq \varphi_{m+1}(x)\varphi_m^2(x)\leq \varphi_{m+1}^2(x).
$$
\par\medskip
(iii) Take $m_0\in {\mathbb N}$ such that $\varphi_{m_0}(\frac{r}{2}) \geq x_0.$ Then, for every $x\geq \frac{r}{2}$ and $m\geq m_0$ we put
$$
\varphi_m(x) = \varphi_{m-m_0}\left(\varphi_{m_0}(x)\right),
$$ where $\varphi_{m_0}(x) \geq x_0.$ Hence
$$
\varphi_m'(x) = \varphi_{m-m_0}'\left(\varphi_{m_0}(x)\right)\cdot \varphi_{m_0}'(x) \leq \left(1+\varphi_{m-m_0}\left(\varphi_{m_0}(x)\right)\right)^2\cdot \varphi_{m_0}'(x).
$$ From (\ref{eq:derivative-as-a-product}) we obtain $\varphi_{m_0}'(x)\leq 2^{m_0}\left(1+\varphi_m(x)\right)^{m_0}.$ Finally, for $p = m_0 + 2$ we conclude
$$
\varphi_m'(x) \leq 2^{m_0}\left(1 + \varphi_m(x)\right)^p\ \ \forall m\geq m_0, x\geq \frac{r}{2}.
$$ Hence we can find $C > 0$ such that
$$
\varphi_m'(x) \leq C\left(1 + \varphi_m(x)\right)^p\ \ \forall m\in {\mathbb N}, x\geq \frac{r}{2}.
$$
\par\medskip
(iv) We now consider $\frac{1}{2}\leq x < \frac{r}{2}$ and select $n_x\geq 1$ with the property that $\varphi_j(x)\geq \frac{r}{2}$ whenever $j\geq n_x$ while $\varphi_j(x) < \frac{r}{2}$ for $0\leq j < n_x.$
\par\medskip
If $m < n_x + 1$ then, from (\ref{eq:derivative-as-a-product}), we get $\varphi_m'(x)\leq r^m.$ Otherwise we decompose
$$
\varphi_m'(x) = \prod_{j=0}^{n_x-1}(2\varphi_j(x))\cdot \prod_{j= n_x}^{m-1}(2\varphi_j(x))
$$ The first factor is dominated by $r^{n_x}\leq r^m,$ while the second one coincides with
$$
2^{m-n_x}\prod_{k= 0}^{m-1-n_x}\varphi_k\left(\varphi_{n_x}(x)\right) = \varphi_{m-n_x}'\left(\varphi_{n_x}(x)\right).
$$ Since $\varphi_{n_x}(x)\geq \frac{r}{2}$ we can use the estimates in (iii) to conclude
$$
\varphi_m'(x) \leq C r^m \left(1+\varphi_{m-n_x}\left(\varphi_{n_x}(x)\right)\right)^p = C r^m \left(1 + \varphi_m(x)\right)^p.
$$
\end{proof}

 \begin{lemma}\label{lem:quadratic2}{\rm Let $\varphi(x) = x^2 + \frac{1}{4}$ be given. Then, for every $r > 1$ and $n\in {\mathbb N}$ there exist $C > 0$ and $p\in {\mathbb N}$ such that
$$
 \left|\varphi_m^{(n)}(x)\right| \leq C r^m\left(1+\varphi_m(x)\right)^p.
 $$
 }
\end{lemma}
\begin{proof} It is enough to consider $x\geq 0.$ The case $n = 1$ is the content of the previous Lemma. Let us now consider $n = 2.$ From (\ref{eq:derivative-as-a-product}) we obtain, for every $x\geq \frac{1}{2},$
\begin{equation}\label{eq:second-derivative}
\varphi_m''(x) = \sum_{j=0}^{m-1}2\varphi_j'(x)\prod_{i\neq j}2\varphi_i(x) = \varphi_m'(x) \sum_{j=0}^{m-1}\frac{\varphi_j'(x)}{\varphi_j(x)}.
\end{equation}
\noindent
Hence
$$
\varphi_m''(x) \leq \varphi_m'(x)\sum_{j=0}^{m-1}2\varphi_j'(x).
$$ We now fix $r > 1$ and take $C > 0$ and $p\in {\mathbb N}$ such that
$$
\left|\varphi_j'(x)\right| \leq C r^j\left(1+\varphi_j(x)\right)^p\ \ \forall j\in {\mathbb N}_0, x\in {\mathbb R}.
$$ Then, for every $x\geq \frac{1}{2},$
$$
\begin{array}{*2{>{\displaystyle}l}}
\left|\varphi_m''(x)\right| & \leq 2 C^2 r^m\left(1+\varphi_m(x)\right)^{2p} \sum_{j=0}^{m-1} r^j \\ & \\ & \leq 2 \frac{C^2}{r-1} r^{2m}\left(1+\varphi_m(x)\right)^{2p}.
 \end{array}
$$ Since $r > 1$ is arbitrary we conclude that for every $r > 1$ there exist $C > 0$ and $q\in {\mathbb N}$ such that $\left|\varphi_m''(x)\right| \leq C r^m \left(1+\varphi_m(x)\right)^{q}$ whenever $x\geq \frac{1}{2}.$ 
\par\medskip
For $n > 2$ we apply (\ref{eq:second-derivative}) to get 
$$
\left|\varphi_m^{(n)}(x)\right| = \left|\left(\varphi_m''\right)^{(n-2)}(x)\right| \leq \sum_{k=0}^{n-2}\binom{n-2}{k}\left|\varphi_m^{(k+1)}(x)\right| \sum_{j=0}^{m-1}\left|\left(\frac{\varphi_j'}{\varphi_j}\right)^{(n-2-k)}(x)\right|.
$$ Now, an application of Leibnitz rule and Fa\`a di Bruno formula permits to proceed by induction in order to prove the desired result for $x\geq \frac{1}{2}$.
\par\medskip
On the other hand, as $\varphi_m^{(n)}$ is increasing in $[0, +\infty)$ then for all $x\in [0, \frac{1}{2}]$ we have
$$
0\leq \varphi_m^{(n)}(x) \leq \varphi_m^{(n)}(\frac{1}{2})\leq C r^m \left(1 + \varphi_m(\frac{1}{2})\right)^p
$$ for some $p$ and $C > 0$ which only depend on $n.$ Since $\varphi_m(\frac{1}{2}) = \frac{1}{2}$ we are done.
\end{proof}

\begin{theorem}{\rm Let $\varphi(x) = x^2 + \frac{1}{4}$ be given. Then $\sigma\left(C_\varphi\right) = \overline{{\mathbb D}}.$
 }
\end{theorem}
\begin{proof} Since $\varphi$ admits a fixed point and $C_\varphi$ is not injective then $\overline{{\mathbb D}}$ is contained in $\sigma\left(C_\varphi\right)$ by Theorem \ref{th:polynomial_fixed_point}. To finish we show that $C_\varphi - \lambda I$ is invertible for every $|\lambda| > 1.$
\par\medskip
(a) $C_\varphi - \lambda I$ is injective for $|\lambda| > 1$. In fact, $C_\varphi(f) = \lambda f$ implies $f\left(\varphi_n(x)\right) = \lambda^n f(x)$ for every $x\in {\mathbb R}$ and $n\in {\mathbb N}.$ Since the left hand side is bounded and $|\lambda| > 1$ then $f(x) = 0$ for every $x\in {\mathbb R}.$
\par\medskip
(b) $C_\varphi - \lambda I$ is surjective for $|\lambda| > 1$. It suffices to show that
$$
\sum_{m=0}^\infty \frac{f\circ \varphi_m}{\lambda^m} 
$$ converges in ${\mathcal S}({\mathbb R})$ for every $f\in {\mathcal S}({\mathbb R})$ and $\lambda\in {\mathbb C}$ with $|\lambda| > 1$. Obviously $|x|\leq 1 + \varphi_m(x)$ for all $x\in {\mathbb R}$ and $m\in {\mathbb N}.$ By Lemmas \ref{lem:quadratic1} and \ref{lem:quadratic2}, $\varphi$ satisfies the hypothesis in Lemma \ref{lem:convergencia_serie} and we conclude.
\end{proof}

Summarizing, we get the following.

\begin{theorem}\label{grau2}{\rm Let $\varphi(x) = a_0 + a_1 x + a_2 x^2$ be a quadratic polynomial with real coefficients and take $c = a_0a_2 + \frac{a_1}{2} - \frac{a_1^2}{4}.$
\begin{itemize}
 \item[(a)] $c > \frac{1}{4}$ implies $\sigma\left(C_\varphi\right) = \left\{0\right\}.$
\item[(b)] $c = \frac{1}{4}$ implies $\sigma\left(C_\varphi\right) = \overline{{\mathbb D}}.$
\item[(c)] $c < \frac{1}{4}$ implies $\sigma\left(C_\varphi\right) = {\mathbb C}.$
\end{itemize}
}
\end{theorem}

\section{Cubic polynomials}

Let us now consider polynomials $\varphi$ of degree $3$ with $\displaystyle\lim_{x\to +\infty}\varphi(x) = +\infty.$

\begin{theorem}\label{th:cubic-positive}{\rm Let $\varphi$ be a polynomial of degree $3$ with positive leading coefficient. Then ${\mathbb C}\setminus \{0\}\subset \sigma(C_\varphi)$ unless $\varphi$ has a fixed point of multiplicity $3$ in which case $\sigma(C_\varphi) = \overline{{\mathbb D}}\setminus\{0\}.$
}
\end{theorem}
\begin{proof}
According to its fixed points the following cases can occur:
\begin{itemize}
\item[(i)] $\varphi$ has three different real fixed points,
\item[(ii)] $\varphi$ has two different real fixed points, one with multiplicity two and the other is simple,
\item[(iii)] $\varphi$ has only  one real fixed point, the other two being complex conjugate numbers,
\item[(iv)] $\varphi$ has one real fixed point of multiplicity 3.
\end{itemize}

In the first three cases, using that $\varphi$ is linearly equivalent to $\psi(x) = -\varphi(-x)$ if necessary, we may apply Theorem \ref{th:simplefixedpoint} (see Example \ref{ex:simplefixedpoint}) to conclude that ${\mathbb C}\setminus\{0\}\subset \sigma(C_\varphi).$ In the forth case, $\varphi$ is linearly equivalent to $\psi(x) = x + x^3.$
\par\medskip
So, to complete the proof we discuss the spectrum of $\varphi(x) = x + x^3.$ From Theorem \ref{th:polynomial_fixed_point} and Proposition \ref{prop:elementary} we already know that $\sigma(C_\varphi)\supset\overline{{\mathbb D}}\setminus\{0\}$ and that $0\notin \sigma(C_\varphi).$

We will show that given  $ n\geq 1$ and $R>1$ there are $C>0$ and $q\in {\mathbb R}$  such that 
\begin{equation}\label{eq:cota-derivada-iterada}
|\varphi_{m}^{(n)}(x)|\leq C R^m(1+\varphi_{m}(x))^q
\end{equation} for each $x\in {\mathbb R}$ and each $m\in {\mathbb N}.$ Since $\varphi$ is an odd function, it suffices to consider $x\geq 0.$ First, we check the inequality (\ref{eq:cota-derivada-iterada}) for the first derivative. Observe that $\varphi_{m+1}'(x)=\varphi_{m}'(x)\varphi'(\varphi_{m}(x)).$ Then, $$\varphi_{m}'(x)=\prod_{k=0}^{m-1}\varphi'(\varphi_{k}(x)).$$

 For $x=0$ we have the inequality with $q=C=1.$ We will proceed by induction on $m.$ For $m=1$ we clearly have the inequality for some $C>0$ and $q=1.$
Also, we have $\varphi'(x)=1+3x^2\geq 1.$ We take $x_0>0$ such that for $x\geq x_0,$ we have $\varphi'(x)<\varphi(x).$ Observe that this implies $x_0>1$ therefore $1+3x^2<(1+x^2)^2.$ As $\varphi(x)>x$ for $x>0,$ we also have 
$$
1+3\varphi_m(x)^2 < \left(1+\varphi_m(x)^2\right)^2$$ for $x\geq x_0.$ Hence, if we assume that for $x\geq x_0$ we have $\varphi_m'(x)\leq (\varphi_m(x))^2,$ we obtain $$\frac{\varphi_{m+1}'(x)}{(\varphi_{m+1}(x))^2}=\frac{\varphi_m'(x)}{(\varphi_m(x))^2}\cdot\frac{1+3\varphi_m(x)^2}{(1+\varphi_m(x)^2)^2}\leq 1.$$ Consequently
$$
\varphi_m'(x) \leq (\varphi_m(x))^2 \ \ \forall m\in{\mathbb N},\ x\geq x_0.
$$
\par\medskip
We now {\it claim} that for each $K>0$ there is $q\in {\mathbb N}$ such that $$\varphi_m'(x)\leq (1 + \varphi_m(x))^q,$$ whenever $x>K.$ In fact, for each $K > 0$ we take $\ell\in {\mathbb N}$ such that $\varphi_\ell(K) > x_0.$ Then, for $x\geq K$ and $m > \ell$ we have
$$
\begin{array}{*2{>{\displaystyle}l}}
 \varphi_m'(x) & = \varphi_{m-\ell}'\left(\varphi_\ell(x)\right)\cdot \varphi_\ell'(x) \\ & \\ & \leq \left(\varphi_{m-\ell}\left(\varphi_\ell(x)\right)\right)^2\cdot \prod_{j=0}^{\ell -1}\left(1 + 3\varphi_j(x)^2\right) \\ & \\ & \leq 3^\ell \left(\varphi_m(x)\right)^2\cdot \left(1 + \varphi_m(x)\right)^{2\ell}.
\end{array}
$$ We take $p\in {\mathbb N}$ such that $3^\ell \leq\left(1 + \varphi(K)\right)^p$ and put $q = 2\ell + p + 2.$ Then 
$$
\varphi_m'(x)\leq (1 + \varphi_m(x))^q$$ for all $x>K.$ The claim is proved. To finish the proof of (\ref{eq:cota-derivada-iterada}) for $n = 1$ we fix $R > 0$ and take $K > 0$ such that $x > 0$ and $\varphi'(x) > R$ imply $x > K.$ For any $x > 0$ let $n_x\in {\mathbb N}$ be the first $n\in {\mathbb N}$ with the property that $\varphi_n(x) > K.$ Then $m \leq n_x$ implies 
$$
\varphi_m'(x) = \prod_{j=0}^{m-1}\varphi'\left(\varphi_j(x)\right) \leq R^m,
$$ while for $m > n_x$ we have
$$
\begin{array}{*2{>{\displaystyle}l}}
 \varphi_m'(x) & = \prod_{j=0}^{n_x-1}\varphi'\left(\varphi_j(x)\right)\cdot \prod_{j=n_x}^{m-1}\varphi'\left(\varphi_j(x)\right)\\ & \\ & \leq R^{n_x} \prod_{j=0}^{m-1-n_x}\varphi'\left(\varphi_j(\varphi_{n_x}(x))\right) \\ & \\ & = R^{n_x} \varphi_{m-n_x}'\left(\varphi_{n_x}(x)\right) \leq R^m \left(1 + \varphi_m(x)\right)^q.
\end{array}
$$

For the second derivative we have, $$\varphi_m''(x)=\sum_{k=0}^{m-1}6\varphi_k(x)^2 \varphi'(x) \prod_{j\neq k}\varphi'(\varphi_j(x))=\varphi_m'(x)\sum_k\varphi_k'(x)\frac{6\varphi_k(x)^2}{1+3\varphi_k(x)^2}.$$ From here, we argue as in the case of $p(x)=x^2+\frac{1}{4}$ to conclude.
\end{proof}

\par\medskip
We observe that each polynomial $\varphi$ of degree $3$ is linearly equivalent to some polynomial of the form
$$
\psi(x) = \pm\ x^3 + Ax + B.
$$

For $\varphi(x) = x^3 + Ax + B$ the spectrum of $C_\varphi$ is already discussed in Theorem \ref{th:cubic-positive}. Next we include some partial results concerning the spectrum of $C_\varphi$  for $\varphi(x) = -x^3 + Ax + B.$ In the special case that $B = 0$ we have a complete characterization.

\begin{proposition}\label{prop:negative-odd}{\rm Let $\eta(x) = -x^3 + Ax$ be given. Then
\begin{itemize}
\item[(a)] $A = -1$ implies $\sigma(C_\eta) = \overline{{\mathbb D}}\setminus\{0\}.$
\item[(b)] $A < 0, A \neq -1,$ implies $\sigma(C_\eta) = {\mathbb C}\setminus \{0\}.$
\item[(c)] $A \geq 0$ implies $\sigma(C_\eta) = {\mathbb C}.$
\end{itemize}
}
\end{proposition}
\begin{proof} Since $\eta$ has fixed points we can apply Theorem \ref{th:polynomial_fixed_point} to conclude that $\overline{{\mathbb D}}\setminus\{0\} \subset \sigma(C_\eta).$ From the fact that $\eta$ is an odd function we have $\eta_2 = \omega_2$ where $\omega(x) = -\eta(x) = x^3 - Ax.$
\par\medskip
(a) In the case $A = -1$ we have $\omega(x) = x^3 + x$ and
$$
\sigma(C_{\eta_2}) = \sigma(C_{\omega_2}) = \overline{{\mathbb D}}\setminus\{0\}.
$$ The last identity follows from Theorem \ref{th:cubic-positive} and spectral theorem. Consequently $\sigma(C_\eta)\subset \overline{{\mathbb D}}\setminus\{0\}.$
\par\medskip
By Proposition \ref{prop:elementary}, in order to show (b) and (c) it suffices to check that ${\mathbb C}\setminus \{0\} \subset\sigma(C_\eta)$ for $A\neq -1.$ We first consider the case $A > -1.$ Then $\omega$ admits a fixed point $a > 0$ such that $\omega'(a) > 1$ and $\omega^{(n)}(a) \geq 0.$ Hence, $\varphi = \eta\circ\eta = \omega\circ\omega$ satisfies the hypothesis of Theorem \ref{th:simplefixedpoint}. Since $\eta(x) < 0 < a$ for all $x\geq a$ we can apply Proposition \ref{rem:composition} to conclude ${\mathbb C}\setminus \{0\} \subset \sigma(C_\eta).$
\par\medskip
In the case $A < -1$ the fixed point $a = 0$ satisfies $\omega'(a) > 1$ and $\omega^{(n)}(a) \geq 0.$ Hence we can proceed as before.
\end{proof}

For polynomials $\varphi(x)=-x^3+Ax+B$ with $B\neq 0$ we can provide some examples.

\begin{proposition}{\rm Let $\psi(x) = x^3 + Ax + B$ be given with $B \neq 0$ and consider $\varphi(x) = -x^3 - Ax - B.$ We assume the $\psi$ has three different (real) fixed points. Then
$$
{\mathbb C}\setminus \{0\} \subset \sigma\left(C_\varphi\right).
$$
}
\end{proposition}
\begin{proof}
 We first assume that $B > 0.$ We have $\varphi_2(x) = \psi_2(x) - 2B$ for every $x\in {\mathbb R}.$ Let $x = \alpha$ be the greatest fixed point of $\psi.$ Then $\psi'(\alpha) > 1$ and $\psi''(x) \geq 0$ for every $x\geq \alpha.$ Then $x = \alpha$ is a fixed point of $\psi_2$ and it satisfies $\psi_2'(\alpha) > 1$ and $\psi_2^{(n)}(x)\geq 0$ for every $x\geq \alpha.$ Since $B > 0$ then the equation
$$
\varphi_2(x) = x,\ \ \mbox{equivalently}\ \ \psi_2(x) = x + 2B
$$ admits a solution $\beta > \alpha.$ Then
$$
\varphi_2'(\beta) = \psi_2'(\beta) \geq \psi_2'(\alpha) > 1,
$$ while
$$
\varphi_2^{(n)}(x) = \psi_2^{(n)}(x) \geq 0 \ \ \forall x\geq \beta.
$$ Consequently
$$
{\mathbb C}\setminus \{0\} \subset \sigma\left(C_{\varphi_2}\right),
$$ from where the conclusion follows after applying Proposition \ref{rem:composition}.
\par
We now consider the case that $B < 0.$ We put $\tilde{\varphi}(x) = -\varphi(-x)$ and $\tilde{\psi}(x) = -\psi(-x).$ Then $\tilde{\psi}(x) = x^3 + Ax - B$ and $\tilde{\varphi}(x) = - \tilde{\psi}(x).$ Since also $\tilde{\psi}$ admits three different fixed points and $\tilde{\psi}(0) > 0$ we conclude
$$
{\mathbb C}\setminus \{0\} \subset \sigma\left(C_{\tilde{\varphi}}\right) = \sigma\left(C_\varphi\right).
$$
\end{proof}

\begin{remark}{\rm Let $\psi(x) = x^3 + Ax + B$ be given with $B > 0.$ If $\psi$ has a unique real fixed point then
$$
\psi(x) = x + \left((x-\alpha)^2 + \beta^2\right)(x-c),
$$ where $c < 0 < \alpha.$
}
\end{remark}

This means that we cannot adapt the previous argument to the case that there is a unique fixed point. Nor can we adapt the argument in the case where there is a simple fixed point and a double fixed point (the condition $B > 0$ forces that the double fixed point is the greatest).

\begin{remark} {\em If $X$ is a locally convex space and $T\in L(X)$, the Waellbrock spectrum $\sigma^\ast(T)$ is defined as the smallest set containing $\sigma(T)$ such that the resolvent mapping $R(\cdot,T): \ \mathbb{C}\setminus \sigma^\ast(T)\to L_b(X)$, $z\mapsto R(z,T)=(zI-T)^{-1}$ is holomorphic (see \cite{Vasilescu}). Here $L_b(X)$ stands for  the space of continuous linear operators on $X$ endowed with the topology of convergence on bounded sets. If $X$ is a Fr\'echet space and $U\subset\mathbb{C}$ is open then $F:U\to L_b(X)$  is holomorphic if and only if the map $U\to \mathbb{C}$, $z\mapsto \langle u, F(z)(x)\rangle$ is holomorphic for every $u\in X'$, $x\in X$ (see \cite[Theorem 1]{GE2004}, \cite[corollary 10, Remark 11]{BFJ2007}). Hence, from Lemma \ref{lem:convergencia_serie} we can get easily that the resolvent map $z\mapsto R(\cdot,C_{\varphi}) $ is holomorphic in $\mathbb{C}\setminus\{0\}$ when $\varphi$ does not have fixed points (Theorem \ref{th:withoutfixedpoints}), and also in $\mathbb{
C}\setminus \overline{\mathbb{D}}$ when $\varphi$ is a polynomial of degree 2 or 3 with a unique fixed point (Theorem \ref{grau2}, Theorem \ref{th:cubic-positive} and Propoposition \ref{prop:negative-odd} a)). In all cases then we have $\sigma^\ast(C_{\varphi})=\overline{\sigma(C_{\varphi})}$.  Contrary to what happens for operators in Banach spaces, where the Waellbrock spectrum equals the spectrum which is always closed, this is not always true for operators defined on Fr\'echet spaces, even when the spectrum is bounded as one can check in \cite[Remark 3.5 (vi)]{abr2013}.}

\end{remark}

\section{Monotone symbols}

We recall that the symbols for $\mathcal{S}({\mathbb R})$ were completely characterized in \cite[Theorem 2.3]{galbis_jorda}. The aim of this section is to provide some information regarding the spectrum of composition operators defined  by monotone symbols. Then, let us assume that the symbol $\varphi$ is strictly monotone and let us denote by $\psi$ its inverse and by $\psi_n$  its $n$-th iterate. For $\lambda \neq 0,$ and $f, \, g \in {\mathcal S}({\mathbb R}),$ the relation $C_\varphi f-\lambda f=g$ implies that  (\ref{eq:pr2}) holds for every $n,$  that is $$ f\left(\varphi_n(x)\right) = \lambda^{n} f(x) + \sum_{k=0}^{n-1}\lambda^{n-1-k}g\left(\varphi_k(x)\right),$$ which implies
\begin{equation}\label{eq:pr4}
\begin{array}{*2{>{\displaystyle}l}}
f(x) & = \lambda^{n} f\left(\psi_n(x)\right) + \sum_{k=0}^{n-1}\lambda^{n-1-k}g\left(\psi_{n-k}(x)\right)\\ & \\ & = \lambda^{n} f\left(\psi_n(x)\right) + \sum_{j=1}^n\lambda^{j-1}g\left(\psi_j(x)\right).
\end{array}
\end{equation}

\begin{proposition}\label{pr:increasing}{\rm Let $\varphi$ be a strictly increasing symbol other than the identity. Then $\sigma(C_\varphi)$ always contains $\{\lambda\in {\mathbb C}: |\lambda|=1\}.$
}
\end{proposition}
\begin{proof}
Let $\lambda$ satisfies $|\lambda| = 1$ and assume that $\lambda\notin \sigma(C_\varphi).$ Then, for every $g\in {\mathcal S}({\mathbb R})$ there is a unique $f\in {\mathcal S}({\mathbb R})$ such that (\ref{eq:pr1}), (\ref{eq:pr2}) and (\ref{eq:pr4}) hold. Now we discuss the following possibilities, covering all possible cases
\begin{itemize}
\item[(1)] $\varphi$ lacks fixed points.
\item[(2)] There exist $a < b$ such that $\varphi(a) = a, \varphi(b) = b$ and $\varphi(x)\neq x$ for every $x\in \left(a,b\right).$
\item[(3)] There exists $a\in {\mathbb R}$ such that $\varphi(a) = a$ and $\varphi(x)\neq x$ for every $x\in \left(a,+\infty\right).$
\item[(4)] There exists $a\in {\mathbb R}$ such that $\varphi(a) = a$ and $\varphi(x)\neq x$ for every $x\in \left(-\infty,a\right).$
\end{itemize}
\par\medskip
(1)
Since, for every $x\in {\mathbb R},$ the sequences $\left(\varphi_n(x)\right)_n$ and $\left(\psi_n(x)\right)_n$ diverge, we obtain from (\ref{eq:pr2}) and (\ref{eq:pr4}),
$$f(x)=-\frac{1}{\lambda}\sum_{k=0}^{\infty}\lambda^{-k}g\left(\varphi_k(x)\right) = \frac{1}{\lambda}\sum_{k=1}^{\infty}\lambda^{k}g\left(\psi_k(x)\right).$$ This implies that
$$
\sum_{k=0}^\infty \lambda^{-k}g(\varphi_{k}(x)) + \sum_{k=1}^\infty \lambda^{k}g(\psi_{k}(x)) = 0
$$ for each $x\in {\mathbb R}.$  This cannot happen if $g$ is a smooth function whose support contains $0$ and is contained in the open interval determined by $\varphi(0)$ and $\psi(0).$
\par\medskip
In the case that the symbol admits some fixed point then $1\in \sigma(C_\varphi),$ so we will take in what follows $\lambda \neq 1.$
\par\medskip
(2) Either
$$
\varphi_n(x)\downarrow a,\ \ \psi_n(x)\uparrow b\ \ \forall x\in (a,b) \ \ (\mbox{if}\ \varphi(x) < x)
$$ or
$$
\varphi_n(x)\uparrow b,\ \ \psi_n(x)\downarrow a\ \ \forall x\in (a,b) \ \ (\mbox{if}\ \varphi(x) > x).
$$ We fix $x_0\in (a,b)$ and define $x_{k+1} = \varphi(x_k).$ Let $I_0$ denote the open interval with extremes $\left(x_0, x_1\right)$ and let $J_0$ be a closed interval contained in $I_0$ and $g$ a smooth function with support contained in $I_0$ such that $g(x) = 1$ for $x\in J_0.$ The identity (\ref{eq:pr1}) implies that $f(a) = f(b) = 0.$ Then (\ref{eq:pr4}) gives
$$
f(x) = \sum_{j=1}^\infty\lambda^{j-1}g\left(\psi_j(x)\right)\ \ \forall x\in (a,b).
$$ Finally we fix $y_0\in J_0,$ define $y_k = \varphi_k(y_0)$ and put $c = \lim_{k\to\infty}y_k$ ($c = a$ or $c = b$). Then $f(y_k) = \lambda^{k-1}g\left(\psi_k(y_k)\right) = \lambda^{k-1}$ and
$$
\lim_{k\to \infty}\left|f(y_k)\right| \neq \left|f(c)\right|,
$$ which is a contradiction.
\par\medskip
(3) In the case $\varphi(x) < x$ for every $x > a$ we have $\varphi_n(x)\downarrow a\ \ \forall x > a$ and we can proceed as in (2) to get a contradiction. We will discuss the case that $\varphi(x) > x$ for every $x > a.$ Then
$$\psi_n(x)\downarrow a\ \ \mbox{while}\ \ \varphi_n(x)\uparrow +\infty\ \ \forall x > a.$$ From (\ref{eq:pr2}) we obtain
$$
f(x) = -\sum_{k=0}^\infty\lambda^{-k-1}g\left(\varphi_k(x)\right)\ \ \forall x > a.
$$ We fix $x_0 > a$ and define $x_{k+1} = \psi(x_k).$ Let $I_0$ denote the open interval $\left(x_1, x_0\right)$ and let $J_0$ be a closed interval contained in $I_0$ and $g$ a smooth function with support contained in $I_0$ such that $g(x) = 1$ for $x\in J_0.$ Finally we fix $y_0\in J_0,$ define $y_k = \psi_k(y_0).$ Then $f(y_k) = -\lambda^{-k-1}g\left(\varphi_k(y_k)\right) = -\lambda^{-k-1}$ and
we can proceed as in case (2) to get a contradiction.
\par\medskip
(4) is analogous to (3).
\end{proof}

For $\varphi(x) = x + e^{-x^2},$ the composition operator $C_\varphi$ is not power bounded but we do not know whether it is mean ergodic or not (see \cite[Remark 1]{fernandez_galbis_jorda}). According to Proposition \ref{pr:increasing}, the spectrum of $C_\varphi$ contains the unit circle.
\par\medskip
We recall that a fixed point $a$ of $\varphi$ is said to be attracting if $|\varphi'(a)|<1$ and repelling if $|\varphi'(a)|>1.$

\begin{proposition}\label{pr:anillo}{\rm Let us assume that $a$ is an attracting  fixed point of the strictly increasing symbol $\varphi.$ Then
$$
\left\{\lambda\in {\mathbb C}:\ \varphi'(a) < |\lambda| < 1\right\} \subset \sigma(C_\varphi).
$$
 }
\end{proposition}
\begin{proof} Let us denote by $\psi$ the inverse of the bijection  $\varphi:[a, +\infty) \to [a,+ \infty).$   Given $\lambda \in {\mathbb D},$  and $f, \, g \in {\mathcal S}({\mathbb R})$, the relation $C_\varphi f-\lambda f=g$, implies that
\begin{equation}\label{eq:anillo}
f(x) = \sum_{j=1}^\infty\lambda^{j-1}g\left(\psi_j(x)\right),\ \ \forall x > a.
\end{equation} We take $\varphi'(a) < \varepsilon < |\lambda|$ and choose $0 < \delta$ so that $\varphi'(x) < \varepsilon$ on $(a, a+\delta).$ For every $x\in (a, a+\delta),$ by the mean value theorem, we have $\varphi(x)<x$ hence the sequence $(\varphi_n(x))_n$ decreases to $a. $ In fact, \begin{equation}\label{mean_value}| \varphi_n(x)-a|\leq \varepsilon^n|x-a|. \end{equation} We fix $a < b < a+\delta$ and let $J$ be a closed interval contained in $(\varphi(b), b)$. We consider a smooth function $g$ whose support is contained in $(\varphi(b),b)$ and such that $g(x) = 1$ for every $x\in J.$ We check that $g$ cannot be in the range of $C_\varphi-\lambda I.$ We proceed by contradiction and assume that there is is $f\in {\mathcal S}({\mathbb R})$ such that $C_\varphi f-\lambda f=g.$ We take $y_0\in J$ and $y_k:=\varphi_k(y_0).$ Then $\psi_j(y_k)\in (\varphi(b),b)$ if and only if $k=j.$ Hence, by (\ref{eq:anillo}), $f(y_k)=\lambda^{k-1}.$
Finally, using (\ref{mean_value}),
$$
\left|\frac{f(y_k)-f(a)}{y_k-a}\right| \geq \frac{|\lambda|^{k-1}}{|y_0-a|\varepsilon^k},
$$ which goes to $\infty$ as $k\to \infty.$ This is a contradiction since $(y_k)_k$  decreases to $a.$
\end{proof}

\begin{corollary}{\rm Let us assume $\varphi(a) = a,$ $\varphi'(a) = 0,$ $\varphi$ strictly increasing. Then
$$
\left\{\lambda\in {\mathbb C}:\  |\lambda| \leq 1\right\} \subset \sigma(C_\varphi).
$$
 }
\end{corollary}

With obvious modifications, one can show

\begin{proposition}\label{pr:anillo_2}{\rm Let us assume that $a$ is a repelling fixed point of $\varphi$ and $\varphi$ strictly increasing. Then
$$
\left\{\lambda\in {\mathbb C}:\ 1\leq |\lambda |<\left|\varphi'(a)\right|\right\} \subset \sigma(C_\varphi).
$$
 }
\end{proposition}

\begin{proposition}\label{pr:circle-decreasing}{\rm Let $\varphi$ be a strictly decreasing symbol. Then
$$
\left\{\lambda\in {\mathbb C}:\ |\lambda| = 1\right\} \subset \sigma(C_\varphi)
$$ if and only if $\varphi\circ\varphi \neq I .$
 }
\end{proposition}
\begin{proof}
 Let us assume $\varphi\circ\varphi \neq I$ and let $a$ denote the unique fixed point of $\varphi.$ We proceed by contradiction, so we assume there is $\lambda\notin \sigma(C_\varphi)$ with $|\lambda| = 1.$ Several possibilities can occur.
 \par\medskip
 (1) There exist $a\leq b < c$ such that $\varphi_2(b) = b, \varphi_2(c) = c$ and $\varphi_2(x) \neq x$ for every $x\in (b,c).$ For every smooth function $g$ whose support is contained in $(b,c)$ there is $f\in {\mathcal S}({\mathbb R})$ with $f\left(\varphi(x)\right) - \lambda f(x) = g(x).$ Then
 $$
 f(x) = \lambda^{2n} f\left(\psi_{2n}(x)\right) + \sum_{j=1}^{2n}\lambda^{j-1}g\left(\psi_j(x)\right).
 $$ Since $\psi\left([b,c]\right) \subset \psi\left([a, +\infty)\right) = (-\infty, a]$ and $\psi_{2k}\left([b,c]\right)\subset [b,c]$ we obtain $\psi_{2k+1}\left([b,c]\right)\subset (-\infty, a].$ Hence
 $$
 f(x) = \lambda^{2n} f\left(\psi_{2n}(x)\right) + \sum_{j=1}^{n}\lambda^{2j-1}g\left(\psi_{2j}(x)\right)\ \ \forall x\in (b,c).
 $$ Now we can argue as in the proof of Proposition \ref{pr:increasing} (case (2)) to get a contradiction.
 \par\medskip
 (2) There exists $b\geq a$ such that $\varphi_2(b) = b$ and $\varphi_2(x)\neq x$ for every $x > b.$ Since $\psi\left([b,+\infty)\right) \subset \psi\left([a, +\infty)\right) \subset (-\infty,a]$ and $\psi_{2k}\left([b, +\infty)\right) \subset [b, +\infty)$ we obtain $\psi_{2k+1}\left([b, +\infty)\right) \subset (-\infty, a].$ Hence
 $$
 f(x) = \lambda^{2n} f\left(\psi_{2n}(x)\right) + \sum_{j=1}^{n}\lambda^{2j-1}g\left(\psi_{2j}(x)\right)\ \ \forall x > b.
 $$ Now we can argue as in the proof of Proposition \ref{pr:increasing} (case (3)) to get a contradiction. The other possibilities can be treated as (1) or (2).
 \par\medskip
 Finally, let us assume that the unit circle is contained in $\sigma(C_\varphi).$ Since $\sigma(C_\varphi^2) \supset \left(\sigma(C_\varphi)\right)^2$ then the unit circle is contained in $\sigma(C_\varphi^2),$ which implies $\varphi\circ\varphi\neq I.$
\end{proof}

If $\varphi \circ \varphi =I$ then $\sigma(C_\varphi) = \{-1,\, 1\}.$ According to \cite[Proposition 3.7]{fernandez_galbis_jorda} and Proposition \ref{pr:circle-decreasing}, the condition $\left\{\lambda\in {\mathbb C}:\ |\lambda| = 1\right\} \subset \sigma(C_\varphi)$ characterizes the decreasing symbols $\varphi$ such that $C_\varphi$ is mean ergodic.

\end{document}